\documentclass[11pt, oneside]{amsart}

\usepackage[british,english]{babel} 
\usepackage{graphics,color,pgf}
\usepackage{epsfig}
\usepackage[ansinew]{inputenc}
\usepackage[all]{xy}
\usepackage{hyperref}
\newdir{ >}{!/8pt/@{}*@{>}}
\usepackage{amssymb, amsmath,amsthm, mathtools, amscd}
\usepackage{mathrsfs}
\usepackage{stmaryrd}
\usepackage[margin=1.5in]{geometry}

\theoremstyle{plain}
\newtheorem{teor}{Theorem}[section]

\newtheorem{cor}[teor]{Corollary}
\newtheorem{prop}[teor]{Proposition}

\theoremstyle{definition}
\newtheorem{deft}[teor]{Definition}

\theoremstyle{remark}
\newtheorem{oss}[teor]{Remark}

\DeclareMathOperator\tr{tr}

\DeclareMathOperator\upG{\textup{G}}

\DeclareMathOperator\upI{\textup{I}}

\DeclareMathOperator\upK{\textup{K}}
\DeclareMathOperator\upL{\textup{L}}

\DeclareMathOperator\bbC{\mathbb{C}}

\DeclareMathOperator\bbH{\mathbb{H}}

\DeclareMathOperator\bbN{\mathbb{N}}

\DeclareMathOperator\bbR{\mathbb{R}}

\DeclareMathOperator\calB{\mathcal{B}}

\DeclareMathOperator\calM{\mathcal{M}}
\DeclareMathOperator\calN{\mathcal{N}}

\DeclareMathOperator\calP{\mathcal{P}}

\DeclareMathOperator\calX{\mathcal{X}}

\DeclareMathOperator\po{\textup{PO}}
\DeclareMathOperator\pu{\textup{PU}}
\DeclareMathOperator\psp{\textup{PSp}}
\DeclareMathOperator\jac{\textup{Jac}}
\DeclareMathOperator\vol{\textup{Vol}}
\DeclareMathOperator\barb{\textup{bar}_{\mathcal{B}}}
\DeclareMathOperator\hypkn{\mathbb{H}^n_{\textup{K}}}
\DeclareMathOperator\hypkm{\mathbb{H}^m_{\textup{K}}}
\DeclareMathOperator\nv{\textup{NV}}

\title[Natural maps for Zimmer's cocycles]{Natural maps for measurable cocycles of compact hyperbolic manifolds}

\author[A. Savini]{A. Savini}
\address{Section de Math\'ematiques, University of Geneva, Rue du Conseil-G\'en\'eral 7-9, 1205 Geneva, Switzerland}
\email{Alessio.Savini@unige.ch}

\thanks{}

\date{\today.\ \copyright{\ A. Savini 2019}.}

\begin{document}

\begin{abstract}
Let $\upG(n)$ be equal to either $\po(n,1),\pu(n,1)$ or $\psp(n,1)$ and let $\Gamma \leq \upG(n)$ be a uniform lattice. Denote by $\hypkn$ the hyperbolic space associated to $\upG(n)$, where $\upK$ is a division algebra over the reals of dimension $d$. Assume $d(n-1) \geq 2$.

In this paper we generalize natural maps to measurable cocycles. Given a standard Borel probability $\Gamma$-space $(X,\mu_X)$, we assume that a measurable cocycle $\sigma:\Gamma \times X \rightarrow \upG(m)$ admits an essentially unique boundary map $\phi:\partial_\infty \hypkn \times X \rightarrow \partial_\infty \hypkm$ whose slices $\phi_x:\hypkn \rightarrow \hypkm$ are atomless for almost every $x \in X$. Then there exists a $\sigma$-equivariant measurable map $F: \hypkn \times X \rightarrow \hypkm$ whose slices \mbox{$F_x:\hypkn \rightarrow \hypkm$} are differentiable for almost every $x \in X$ and such that $\jac_a F_x \leq 1$ for every $a \in \hypkn$ and almost every $x \in X$. This allows us to define the natural volume $\nv(\sigma)$ of the cocycle $\sigma$. This number satisfies the inequality $\nv(\sigma) \leq \vol(\Gamma \backslash \hypkn)$. Additionally, the equality holds if and only if $\sigma$ is cohomologous to the cocycle induced by the standard lattice embedding $i:\Gamma \rightarrow \upG(n) \leq \upG(m)$, modulo possibly a compact subgroup of $\upG(m)$ when $m>n$. 

Given a continuous map $f:M \rightarrow N$ between compact hyperbolic manifolds, we also obtain an adaptation of the mapping degree theorem to this context.
\end{abstract}

\maketitle

\section{Introduction}

Let $(M,g)$ be a compact Riemannian $n$-manifold which admits a locally symmetric Riemannian metric $g_0$. The minimal entropy conjecture states that the functional given by the volume entropy suitably rescaled by the volume of $g$, that is $(h_{\textup{vol}}(g))^n\vol(M,g)$, is minimized uniquely by the locally symmetric structure(s) on $M$, up to a possible homothety (\cite{Gromov83}). A positive answer for surfaces is given by Besson, Curtois and Gallot \cite{bcg91}, but in this case one has to notice that there exist infinitely many inequivalent hyperbolic structures. On the contrary, when $n \geq 3$, there exist a unique locally symmetric structure on a compact manifold of rank one. For those, a proof of the conjecture is given by Besson, Courtois and Gallot \cite{bcg95,bcg96,bcg98} with the introduction of the so called \emph{natural maps}, whereas in the higher rank case the conjecture is still open. 

Natural maps revealed a very powerful tool and, for this reason, they have been used later in the study of many other different problems. For instance Boland, Connell and Souto \cite{boland05} apply them to the study of volume rigidity of non-uniform real hyperbolic lattices. Another example is given by Francaviglia and Klaff. In \cite{franc06:articolo,franc09} the authors exploit the notion of natural map to study the rigidity of representations of real hyperbolic lattices. Given a torsion-free lattice $\Gamma \leq \po^\circ(n,1)$ and a representation $\rho:\Gamma \rightarrow \po(m,1)$ with $m \geq n \geq 3$, they show the existence of a smooth $\rho$-equivariant map $F: \bbH^n_{\bbR} \rightarrow \bbH^m_{\bbR}$ which satisfies $\jac _a F \leq 1$ for every $a \in \bbH^n_{\bbR}$. Additionally, when $\Gamma$ is non-uniform, they introduce a family of differentiable maps $F^\varepsilon:\bbH^n_{\bbR} \rightarrow \bbH^m_{\bbR}$ depending on $\varepsilon >0$, which are still $\rho$-equivariant, properly ending and satisfy $\jac_a F^\varepsilon \leq (1+\varepsilon)$ for every $a \in \bbH^n_{\bbR}$ (the properly ending property can be interpreted as a compatibility condition of the map on the peripheral subgroups of $\Gamma$). The constructions described above allow to introduce the notion of volume $\vol(\rho)$ of the representation $\rho$ by considering the infimum over all the possible volumes $\vol(D)$, where $D$ is a smooth $\rho$-equivariant map (which properly ends in the \emph{non-unifom} case). 

Volume of representations remains unchanged under the conjugation by an element $g \in \po(m,1)$ and it satisfies a Milnor-Wood type inequality. Indeed we have that $\vol(\rho) \leq \vol(\Gamma \backslash \bbH^n_{\bbR})$ and the equality is attained if and only if the representation is conjugated by an element of $\po(m,1)$ to the standard lattice embedding $i:\Gamma \rightarrow \po(n,1) \leq \po(m,1)$, modulo possibly a compact subgroup when $m > n$. Here $\po(n,1)$ is realized as a subgroup of $\po(m,1)$ via the upper-left corner embedding. 

Notice that when $n=m=3$ the volume of a representation coincides with the definition given independently by Dunfield \cite{dunfield:articolo} and by Francaviglia \cite{franc04:articolo} in terms of pseudo-developing maps (both definitions generalize the notion of volume of a hyperbolic structure reported for instance in \cite{neumann:zagier}). It is worth mentioning that similar rigidity results have been obtained by Bucher, Burger and Iozzi \cite{bucher2:articolo} in the case $n=m$. However their approach to the problem is completely different and their definition of volume of representations relies on the study of the bounded cohomology groups of $\po(n,1)$. 

In the context of rank-one torsion-free lattices, similar questions have been studied also for complex and quaternionic lattices. Given a non-uniform torsion-free lattice $\Gamma \leq \pu(n,1)$ and a representation $\rho:\Gamma \rightarrow \pu(m,1)$ with $m \geq n \geq 2$, Koziarz and Maubon \cite{koziarz:maubon} prove a rigidity result analogous to the one described above but using the theory of harmonic maps. In \cite{BIcartan} Burger and Iozzi obtain the same statement for both uniform and non-uniform lattice using jointly bounded cohomology and $\upL^2$-cohomology. As far as it concerns the study of quaternionic lattices, it is worth mentioning the superrigidity result that Corlette obtains in \cite{corlette92}. 

Recently the author has shown in \cite{savini:articolo,savini2:articolo} a stronger rigidity phenomenon for the volume function. Indeed volume of representations of any rank-one torsion-free \emph{non-uniform} lattice is rigid at the ideal points of the character variety. For instance, if $\Gamma$ is a torsion-free \emph{non-uniform} real hyperbolic lattice, the character variety $X(\Gamma,\po(m,1))$ is an algebraic set of positive dimension and a divergent sequence of representations cannot eventually maximize the volume. The same result can be suitably adapted to the context of complex and quaternionic lattices. To sum up one could say that for rank-one torsion-free lattices the volume of representations is asymptotically rigid.

As already said, the minimal entropy conjecture is still open in the higher rank case. However it is worth mentioning some efforts which move towards the direction of a proof. In \cite{connellfarb1} Connell and Farb succeed in extending the construction of natural maps to lattices in products of rank-one Lie groups of non-compact type. The key point is that they prove an estimate on the Jacobian of the natural map which is still sharp. Similarly they obtain a uniform, but not sharp, Jacobian estimate for more general higher rank symmetric spaces (\cite{connellfarb2}).

Other interesting applications of natural maps have been found for foliations of Riemannian manifolds with locally symmetric negatively curved leaves (\cite{boland:connell}) and for Finsler/Benoist manifolds (\cite{boland:newberger,bray,savini_conv_proj}). 

In this paper we want to extend the notion of natural map to the setting of Zimmer's cocycles theory in order to study rigidity phenomena. Recently this kind of study has been developed by the author and Moraschini using the theory of bounded cohomology (see for instance \cite{savini3:articolo,savini:surface,moraschini:savini,moraschini:savini:2}). Here we want to give a differentiable approach to this subject. More precisely, denote by $\upG(n)$ either $\po(n,1),\pu(n,1)$ or $\psp(n,1)$ and let $\Gamma \leq \upG(n)$ be a torsion-free \emph{uniform} lattice. Since it is well-known that the Riemannian symmetric space associated to $\upG(n)$ is a hyperbolic space on a suitable division algebra $\upK$, we denote it by $\hypkn$. If we denote by $d=\textup{dim}_{\bbR} \upK$ the real dimension of the division algebra $\upK$, we will need to assume $d(n-1) \geq 2$. Fix now a standard Borel probability $\Gamma$-space $(X,\mu_X)$ without atoms. Suppose $m \geq n$ and consider a Zimmer's cocycle $\sigma:\Gamma \times X \rightarrow \upG(m)$ with essentially unique $\sigma$-equivariant measurable map $\phi:\partial_\infty \hypkn \times X \rightarrow \partial_\infty \hypkm$. The boundary map $\phi$ allows us to define for almost every $x \in X$ the \emph{slice} $\phi_x:\partial_\infty \hypkn \rightarrow \partial_\infty \hypkm$ given by $\phi_x(\xi):=\phi(\xi,x)$. Notice that for almost every $x \in X$ the slice $\phi_x$ is measurable since $X$ is standard Borel (\cite[Lemma 2.6]{fisher:morris:whyte}). Supposing that for almost every slice $\phi_x$ the push-forward of the Patterson-Sullivan measure is atom-free (hence the slice is \emph{atomless}), we can apply the barycenter construction to get the desired natural map. In this way we obtain the following:

\begin{teor}\label{teor:natural:map}
Let $\upG(n)$ be either $\po(n,1),\pu(n,1)$ or $\psp(n,1)$ and denote by $\hypkn$ the associated hyperbolic space over the division algebra $\upK$ of dimension $d=\dim_{\bbR} \upK$. Let $\Gamma \leq \upG(n)$ be a torsion-free uniform lattice and fix $(X,\mu_X)$ a standard Borel probability $\Gamma$-space. Suppose $d(n-1) \geq 2$ and take $m \geq n$. Given a measurable cocycle $\sigma:\Gamma \times X \rightarrow \upG(m)$, assume there exists an essentially unique boundary map $\phi:\partial_\infty \hypkn \times X \rightarrow \partial_\infty \hypkm$ with atomless slices. Then there exists a measurable map $F: \hypkn \times X \rightarrow \hypkm$ which is $\sigma$-equivariant. Additionally for almost every $x \in X$ the slice $F_x:\hypkn \rightarrow \hypkm$ is smooth and we have
$$
\jac _a F_x \leq 1 \ , 
$$
for every $a \in \hypkn$. The equality is attained if and only if the map $D_aF_x: T_a \hypkn \rightarrow T_{F_x(a)} \hypkm$ is an isometric embedding. 
\end{teor}

Since the map $F:\hypkn \times X \rightarrow \hypkm$ is a clear generalization of the natural map defined by Besson, Courtois and Gallot to the context of Zimmer's cocycles, we are going to say that $F$ is \emph{the natural map associated to the cocycle $\sigma$}. Even if the condition on the slices of the boundary map may seem quite restrictive, natural maps exist for measurable cocycles coming from couplings, as shown by Bader, Furman and Sauer \cite[Lemma 3.6]{sauer:articolo} and by the author \cite[Lemma 3.1]{savini4}. 

We are also going to define the notion of volume associated to $\sigma$. Our definition will differ from the one given by Francaviglia and Klaff for representations. Indeed here we are going to concentrate our attention only to the natural map associated to a fixed cocycle, without taking any infimum over all the possible volumes of equivariant maps.

Given any $\sigma$-equivariant measurable map $\Phi: \hypkn \times X \rightarrow \hypkm$ with differentiable slice $\Phi_x: \hypkn \rightarrow \hypkm$ for almost every $x \in X$, we can consider the volume form associated to the pullback metric $(\omega_x)_a(u_1,\ldots,u_p):=\sqrt{ \det g_m(D_a\Phi_x(u_i),D_a\Phi_x(u_j))}$, where $u_1,\ldots,u_p \in T_a\hypkn$ and $g_m$ is the standard Riemannian metric on $\hypkm$. When $\Phi$ satisfies the \emph{essential boundedness} condition we obtain a measurable family of differential forms $\{ \omega_x \}_{x \in X}$ on $\hypkn$ (see Section \ref{sec:vol:coc}). Hence by considering its integral over $X$, we obtain a differential form on $\hypkn$ which is $\Gamma$-invariant by the equivariance of the map $\Phi$. Thus we have a differential form on $\Gamma \backslash \hypkn$ and we can take its integral. This number will be the \emph{volume associated to the measurable map $\Phi$}. If we specialize to the case of the natural map, which is essentially bounded, we call it the \emph{natural volume $\nv(\sigma)$ of $\sigma$}. 

Clearly the natural volume of a cocycle will be invariant by the conjugation action of $\upG(m)$ on the space of cocycles. Moreover this volume satisfies a Milnor-Wood inequality type similar to the one obtained by Bucher, Burger and Iozzi \cite{bucher2:articolo} for representations, by Bader, Furman and Sauer \cite{sauer:articolo} for self-couplings and by the author and Moraschini \cite{moraschini:savini,moraschini:savini:2} for cocycles. Notice that the result obtained in \cite{moraschini:savini} are valid for $n=m$ whereas here we can consider also the case $m >n$. 

The Milnor-Wood inequality obtained here will be crucial to prove the following rigidity result:

\begin{teor}\label{teor:rigidity:cocycle}
Let $\upG(n)$ be either $\po(n,1),\pu(n,1)$ or $\psp(n,1)$ and denote by $\hypkn$ the associated hyperbolic space over the division algebra $\upK$ of dimension $d=\dim_{\bbR} \upK$. Let $\Gamma \leq \upG(n)$ be a torsion-free uniform lattice and fix $(X,\mu_X)$ a standard Borel probability $\Gamma$-space. Suppose $d(n-1) \geq 2$ and take $m \geq n$. Given a measurable cocycle $\sigma:\Gamma \times X \rightarrow \upG(m)$, assume there exists an essentially unique boundary map $\phi:\partial_\infty \hypkn \times X \rightarrow \partial_\infty \hypkm$ with atomless slices. Then 
$$
\nv(\sigma) \leq \vol(\Gamma \backslash \hypkn) \ ,  
$$
and the equality holds if and only if $\sigma$ is cohomologous to the cocycle induced by the standard lattice embedding $i:\Gamma \rightarrow \upG(n) \leq \upG(m)$, modulo possibly a compact subgroup of $\upG(m)$ when $m>n$. Here $\upG(n)$ is realized into $\upG(m)$ via the upper-left corner embedding. 
\end{teor}

First notice that for both Theorem \ref{teor:natural:map} and Theorem \ref{teor:rigidity:cocycle} the hypothesis of uniformity of the lattice $\Gamma$ is not crucial. Hence the same results can be generalized also to non-uniform lattices. A suitable variation of Theorem \ref{teor:rigidity:cocycle} has been exploited by the author \cite{savini4} to show that the group $\pu(n,1)$ is $1$-taut in the sense of Bader, Furman and Sauer \cite{sauer:articolo}. That problem was an open conjecture whose proof allows us to classify finitely generated groups that are integrable measure equivalent to a complex hyperbolic lattice. 

The proof of Theorem \ref{teor:rigidity:cocycle} relies on the sharpness on the estimate of the Jacobian of the slices of $F$. More precisely one can see that if the volume of $\sigma$ is maximal then for almost every $x \in X$ the Jacobian must satisfy $\jac_a F_x=1$ for almost every $a \in \hypkn$. In particular the slice $F_x$ coincides essentially with a totally geodesic embedding and hence it is essentially equal to an element $F_x=f(x) \in \upG(m)$. In this way we obtain a map $f: X \rightarrow \upG(m)$ whose measurability is guaranteed by \cite[Lemma 2.6]{fisher:morris:whyte}. We conclude the proof applying the strategy exposed in \cite[Proposition 3.2]{sauer:articolo} and adding a measurable function into a compact subgroup of $\upG(m)$, when $m >n$.  

The notion of volume is useful also to study the mapping degree of continuous maps between closed hyperbolic manifolds of the same dimension. The mapping degree theorem, first stated by Kneser \cite{kneser} for surfaces and then extended by Thurston \cite{Thurston} in the higher dimensional case, states that given a continuous map $f:M \rightarrow N$ between closed real hyperbolic manifolds of the same dimension, it must hold that
$$
|\deg(f)| \leq \frac{\vol(M)}{\vol(N)} \ .
$$
Additionally, the strict version \cite[Theorem 6.4]{Thurston} of the theorem characterizes local isometries as those maps satisfying the equality. 

Several proofs of the mapping degree theorem have been given so far. For instance Thurston \cite{Thurston} and Gromov \cite{Grom82} used $\ell^1$-homology and the notion of simplicial volume. Besson, Courtois and Gallot \cite{bcg95,bcg96,bcg98} obtained a proof based on their concept of natural map (generalized later by Connell and Farb \cite{connellfarb1,connellfarb2} to the higher rank case). For real hyperbolic manifolds, it is worth mentioning the approach of Bucher, Burger and Iozzi \cite{bucher2:articolo} based on the study of bounded cohomology groups of $\po(n,1)$. Similarly the author and Moraschini \cite{moraschini:savini} obtained an analogous proof by studying the notion of maximal Zimmer's cocycles. 

The interest in the relation between the mapping degree of continuous maps and the volume of manifolds led to a rich and fruitful literature \cite{LK:degree,Neof1,Neof2,FM:ideal}. Derbez, Liu, Sun and Wang \cite[Proposition 3.1]{derbez:liu:sun:wang} were able to express the volume of the pullback of a representation $\rho$ along a continuous map $f$ as the product of the mapping degree of $f$ with the volume of $\rho$. The same has been done in \cite{moraschini:savini} in the case of maximal cocycles. Here we want to generalize this result to the context of measurable maps with smooth slices which are equivariant with respect a fixed measurable cocycle. Given continuous map $f:M \rightarrow N$ and a measurable equivariant map $\Phi:\hypkn \times X \rightarrow \hypkm$, one can suitably define the notion of \emph{pullback map $f^\ast \Phi$ along the continuous map $f$} (see Section \ref{sec:vol:degree}). Then we have the following version of the mapping degree theorem.

\begin{prop}\label{prop:map:deg}
Let $\Gamma,\Lambda \leq \upG(n)$ be two torsion-free uniform lattices. Set $M=\Gamma \backslash \hypkn$, $N=\Lambda \backslash \hypkn$ and let $f:M \rightarrow N$ be a continuous map with non-vanishing degree. Fix a standard Borel probability $\Lambda$-space $(X,\mu_X)$ and consider a measurable cocycle $\sigma:\Lambda \times X \rightarrow \upG(m)$. Given any measurable $\sigma$-equivariant map $\Phi:\hypkn \times X \rightarrow \hypkm$ with essentially bounded smooth slices we have that 
$$
|\deg(f)| \leq \frac{\vol(f^\ast \Phi)}{\vol(\Phi)} \ .
$$
Additionally if $f$ is homotopic to a local isometry then the equality is attained. 
\end{prop}

The proof of the latter statement will rely essentially on both the co-area formula and on Thurston's strict version of mapping degree theorem. If in the previous proposition we consider the particular case when $\Phi$ is the natural map associated to $\sigma$, then we get an inequality which relates the volume of $f^\ast\Phi$ and the natural volume of $\sigma$, that is:
$$
|\deg(f)| \leq \frac{ \vol(f^\ast \Phi)}{\nv(\sigma)} \ .
$$  

It is worth noticing that a priori we do not know if $f^\ast \Phi$ is the natural map associated to $f^\ast \sigma$, so we cannot push any further our reasoning. Nevertheless, the above estimate allows us to characterize maps homotopic to local isometries in terms of naturally maximal cocycles, that is cocycles with maximal natural volume (compare with \cite[Proposition 1.3]{moraschini:savini}).

\subsection*{Acknowlegdements} I am grateful to the anonymous referee for her/his suggestions that allowed me to improve the quality of the paper. 

\subsection*{Plan of the paper} 

Section \ref{sec:def:res} is devoted to recall the main definitions and results that we will need in the paper. More precisely, in Section \ref{subsec:zimmer:cocycle} we briefly recall Zimmer's cocycles theory. We move then to the definition of barycenter for atom-free probability measures on the boundary at infinity $\partial_\infty \hypkn$, described in Section \ref{subsec:bary:meas}. Then in Section \ref{subsec:patterson:sullivan} we expose the notion of the Patterson-Sullivan density associated to a lattice and the construction of the Besson-Courtois-Gallot natural map. The crucial definition of natural map associated to a Zimmer's cocycle appears in Section \ref{sec:natural:map}, where we also discuss all of its properties. We show that it can be suitably interpreted as a generalization of the natural map for representations (Proposition \ref{prop:natural:rep}). Additionally we show how natural maps vary along the $\upG(m)$-cohomology class (Proposition \ref{prop:natural:cohomology}).  

We move to Section \ref{sec:vol:coc}. Here we introduce the notion of volume of a measurable equivariant map with essentially bounded differentiable slices and subsequently the notion of natural volume of measurable cocycles (see Definitions \ref{def:volume:map} and \ref{def:volume:coc}). Then the main rigidity result is proved.

We conclude with Section \ref{sec:vol:degree} where we prove our version of mapping degree theorem and some comments about natural volume and naturally maximal cocycles follow.

\section{Preliminary definitions and results} \label{sec:def:res}

In this section we are going to recall briefly all the notions we will need in the paper in order to define the natural map associated to a Zimmer's cocycle. For this reason we first discuss the notion of measurable cocycle and we will see how representation theory fits into this wider context. Then we focus our attention on some elements of boundary theory, in particular on the notion of generalized boundary map. We conclude this digression by talking about the Patterson-Sullivan measures and the barycenter construction. We will need both to construct our natural map. Indeed we are going to apply the barycenter to the push-foward of the Patterson-Sullivan measures with respect to the slices of the boundary map associated to a measurable cocycle. 

\subsection{Zimmer's cocycle theory}\label{subsec:zimmer:cocycle}

For all the material in the following section we mainly refer to the work of both Furstenberg \cite{furst:articolo73,furst:articolo} and Zimmer \cite{zimmer:preprint,zimmer:libro}. 

Let $G,H$ be two locally compact second countable groups endowed both with their natural Haar measure. Consider a standard Borel measure space $(X,\mu)$ on which $G$ acts via measure preserving transformations. We are going to call the space $(X,\mu)$ satisfying the hypothesis above a \emph{standard Borel measure $G$-space}. If additionally $(X,\mu)$ is a probability space without atoms we are going to say that $(X,\mu)$ is a \emph{standard Borel probability $G$-space}. 

Given another measure space $(Y,\nu)$, we are going to denote by $\textup{Meas}(X,Y)$ the \emph{space of measurable functions from $X$ to $Y$}, endowed with the topology of the convergence in measure. 

\begin{deft}\label{def:zimmer:cocycle}
Let $\sigma:G \times X \rightarrow H$ be a measurable function. We call $\sigma$ a \emph{measurable cocycle} (or \emph{Zimmer's cocycle} or simply \emph{cocycle}) if the associated map 
$$
\sigma: G \rightarrow \textup{Meas}(X,H), \hspace{5pt} g \mapsto \sigma(g,\cdot) \ , 
$$
is continuous and it holds 
\begin{equation}\label{equation:zimmer:cocycle}
\sigma(g_1g_2,x)=\sigma(g_1,g_2.x)\sigma(g_2,x) \ , 
\end{equation}
for every $g_1,g_2 \in G$ and almost every $x \in X$. 
\end{deft}

In the previous definition we preferred to stress the action of $G$ on $X$ using the dot, but from now on we will omit this symbol. 

At a first sight the notion of measurable cocycle might seem quite mysterious to the reader who is not familiar with this theory. One could interpret Equation (\ref{equation:zimmer:cocycle}) either as a suitable generalization of the chain rule for derivatives or as the classic Eilenberg-MacLane condition for Borel $1$-cocycle (see \cite{feldman:moore,zimmer:preprint}). The latter interpretation comes from viewing the cocycle as an element $\sigma \in \textup{Meas}(G,\textup{Meas}(X,H))$. Following the line of this interpretation it is natural to define also the notion of cohomologous cocycles.

\begin{deft}\label{def:coom:cocycle}
Let $\sigma_1,\sigma_2:G \times X \rightarrow H$ be two measurable cocycles and let $f:X \rightarrow H$ be a measurable function. Then the cocycle defined by
\begin{equation}\label{def:twisted:cocycle}
\sigma^f:G \times X \rightarrow H, \hspace{5pt} \sigma^f(g,x):=f(gx)^{-1}\sigma(g,x)f(x) \ ,
\end{equation}
is the \emph{twisted cocycle associated to $\sigma$ and $f$}. The two cocycles $\sigma_1$ and $\sigma_2$ are \emph{cohomologous} (or \emph{equivalent}) if 
$$
\sigma_2=\sigma^f_1 \ ,
$$
for some measurable function $f:X \rightarrow H$. 
\end{deft}

Measurable cocycles are quite ubiquitous in mathematics. Indeed one can find them in several different contexts, like in differential geometry (the differentiation cocycle, see \cite[Example 4.2.2]{zimmer:libro}) or in measure theory (the Radon-Nikodym cocycle, see \cite[Example 4.2.3]{zimmer:libro}). In our case, we are going to focus on a large family of cocycles coming from representation theory. 

\begin{deft}\label{def:cocycle:representation}
Let $\rho:G \rightarrow H$ be a continuous representation and let $(X,\mu)$ be any standard Borel measure $G$-space. We define the \emph{cocycle associated to the representation $\rho$} as follows
$$
\sigma_\rho:G \times X \rightarrow H, \hspace{5pt} \sigma_\rho(g,x):=\rho(g) \ , 
$$
for every $g \in G$ and almost every $x \in X$. 
\end{deft}

From the definition it should be clear that any continuous representation naturally determines a measurable cocycle once we have fixed a suitable standard Borel measure $G$-space. Notice that even if the variable $x \in X$ does not arise in the definition above, the cocycle $\sigma_\rho$ actually depends on both the representation $\rho$ and the space $X$. Nevertheless we prefer to omit the latter dependence to avoid a heavy notation. Notice also that when $G$ is a discrete group, any representation is continuous and hence we can always define an associated measurable cocycle.

Another key tool we will need later is the concept of boundary map associated to a measurable cocycle. Here we are going to introduce directly the notion of \emph{generalized boundary map}, even if we will not need it in its full generality. Assume first that $G$ admits a Furstenberg-Poisson boundary $B(G)$ (see \cite{furstenberg:annals} for a precise definition). A well-known example of this situation is when $G$ is a center-free semisimple Lie group without compact factors. In this case the Furstenberg-Poisson boundary $B(G)$ can be identified with the homogeneous space $G/P$, where $P \leq G$ is any minimal parabolic subgroup. If we denote by $\calX_G$ the Riemannian symmetric space associated to $G$, then usually the Furstenberg-Poisson boundary $B(G)$ is strictly contained in the boundary at infinity $\partial_\infty \calX_G$. However, when $G$ has real rank one, the two coincide since it holds $$\textup{codim}_{\partial_{\infty} \calX_G} B(G)=\textup{rank}_{\mathbb{R}}(G)-1 \ . $$ 

Endow now $B(G)$ with its natural Borel structure coming from the Haar sigma-algebra on $G$ and suppose that $H$ acts measurably on a compact completely metrizable space $Y$. 

\begin{deft}\label{def:boundary:map}
Let $\sigma:G \times X \rightarrow H$ be a measurable cocycle. A measurable map $\phi:B(G) \times X \rightarrow Y$ is \emph{$\sigma$-equivariant} if it holds
$$
\phi(g \xi, gx)=\sigma(g,x)\phi(\xi,x) \ , 
$$
for all $g \in G$ and almost every $\xi \in B(G)$ and $x \in X$. A \emph{generalized boundary map} (or simply \emph{boundary map}) is the datum of a measurable map $\phi$ which is $\sigma$-equivariant. 
\end{deft}

The existence and the uniqueness of boundary maps for a measurable cocycle $\sigma$ usually rely on the properties of the cocycle. For instance the \emph{proximality} of the cocycle guarantees the existence of such a map. We are not going to define proximality here and we refer the reader to \cite{furst:articolo} for a detailed exposition. 

Since we introduce the notion of cohomologous cocycles, it is natural to show how boundary maps change along the $H$-cohomology class of a fixed cocycle.

\begin{deft}\label{def:twist:bound:map}
Let $\sigma:G \times X \rightarrow H$ be a measurable cocycle with (generalized) boundary map $\phi:B(G) \times X \rightarrow Y$. Let $f:X \rightarrow H$ be a measurable map. The \emph{boundary map associated to the twisted cocycle $\sigma^f$} is given by 
$$
\phi^f:B(G) \times X \rightarrow Y, \hspace{5pt} \phi^f(\xi,x)=f(x)^{-1}\phi(\xi,x) \ ,
$$
for almost every $\xi \in B(G)$ and $x \in X$.
\end{deft}

We conclude this section by introducing the notion of slice associated to a boundary map. We will need this definition since we are going to assume the atomless property of the slices of a boundary map in order to construct our natural map.

\begin{deft}\label{def:boundary:slice}
Let $\sigma:G \times X \rightarrow H$ be a measurable cocycle and let $\phi:B(G) \times X \rightarrow Y$ be a (generalized) boundary map. For almost every $x \in X$ we define the \emph{slice associated to the point $x$} as follows
$$
\phi_x:B(G) \rightarrow Y, \hspace{5pt} \phi_x(\xi):=\phi(\xi,x) \ ,
$$
for almost every $\xi \in B(G)$. 
\end{deft}

By the equivariance of the map $\phi$ the slices are related by the following equation 
\begin{equation}\label{equation:gamma:slice}
\phi_{gx}(g \  \cdot \ )=\sigma(g,x)\phi_x(\cdot)
\end{equation}
for every $g \in G$ and almost every $x \in X$.

It is worth noticing that for almost every $x \in X$ the slice $\phi_x$ is measurable. Indeed, since we assumed that $X$ is a standard Borel space, we know that the function $\widehat{\phi}:X \rightarrow \textup{Meas}(B(G),Y), \hspace{5pt} \widehat{\phi}(x):=\phi_x,$ is well-defined and measurable by \cite[Lemma 2.6]{fisher:morris:whyte}. 

\subsection{Barycenter of a probability measure} \label{subsec:bary:meas}

In this section we are going to recall the barycenter construction introduced by Douady and Earle in their paper \cite{douady:earle}. As in the classic case of Besson, Courtois and Gallot this machinery will be crucial to construct our natural map. 

Before giving the definition of the barycenter, we first need to recall the notion of Busemann function. Let $b \in \hypkn$ a fixed basepoint. The \emph{Busemann function pointed at $b$} is the function given by
$$
\beta_b: \hypkn \times \partial_\infty \hypkn \rightarrow \bbR, \hspace{5pt} \beta_b(a,\xi):=\lim_{t \to \infty} d(a,c(t))-d(b,c(t)) \ ,
$$
where $c:[0,\infty) \rightarrow \hypkn$ is a geodesic ray starting at $c(0)=b$ and ending at $\xi$. The distance $d$ is the one induced by the standard Riemannian structure on the hyperbolic space $\hypkn$. Fix now a basepoint $o \in \hypkn$. By an abuse of notation we will use the same symbol for the basepoint $o$ in hyperbolic spaces of different dimensions. We are going to denote by $\beta_o(x,\xi)$ the Busemann function pointed at the basepoint $o \in \hypkn$.

Given any topological space $X$, denote by $\calM^1(X)$ the space of positive probability measures on $X$. Consider now any positive probability measure $\nu \in \calM^1(\partial_\infty \hypkn)$ on the boundary at infinity of the hyperbolic space. For our purposes, it will be sufficient to consider the case when $\nu$ does not contain any atom. A crucial property of Busemann functions is given by their \emph{convexity} (see \cite[Chapter 8]{atha:papa}). Using this property we get immediately that the function 
$$
\Lambda_\nu: \hypkn \rightarrow \bbR, \hspace{5pt} \Lambda_\nu(x):=\int_{\partial_\infty \hypkn} \beta_o(x,\xi)d\nu(\xi) \ ,
$$
is convex. Moreover, since the following condition holds
$$
\lim_{x \to \partial_\infty \hypkn} \Lambda_\nu(x)=\infty \ ,
$$
the function $\Lambda_\nu$ attains its minimum inside $\hypkn$. The uniqueness of such minimum is guaranteed by the fact that $\nu$ does not have any atom. We refer the reader to either \cite[Appendix A]{bcg95} or \cite[Proposition 3.7]{bcg98}.

\begin{deft}\label{def:bar:meas}
Let $\nu \in \calM^1(\partial_\infty \hypkn)$ be a positive probability measure which does not contain any atom. The \emph{barycenter of the measure $\nu$} is defined as
$$
\barb(\nu):=\textup{argmin} (\Lambda_\nu) \ .
$$
Notice that the subscript $\calB$ we used in the definition emphasizes the dependence of the barycenter construction on the Busemann functions. 
\end{deft}

Under the assumptions we made on the probability measure $\nu$, its barycenter $\barb(\nu)$ will be a point in $\hypkn$ which satisfies the following properties
\begin{description}

\item[(i)] the barycenter is continuous with respect to the weak-${}^*$ topology on the space $\calM^1(\partial_\infty \hypkn)$. More precisely if $\nu_k \rightarrow \nu$ in the weak-${}^*$ topology it holds that 
$$
\lim_{k \to \infty} \barb(\nu_k)=\barb(\nu) \ ;
$$

\item[(ii)] the barycenter is $\upG(n)$-equivariant. Recall first that $\upG(n)$ is the isometry group of the Riemannian symmetric space $\hypkn$. Then for every $g \in \upG(n)$ we have that
$$
\barb(g_*\nu)=g\barb(\nu) \ ,
$$
for every $\nu \in \calM^1(\partial_\infty \hypkn)$. The symbol $g_*\nu$ stands for the push-forward measure of $\nu$ with respect to the isometry $g$; 

\item[(iii)] the barycenter satisfies an implicit equation given by
\begin{equation}\label{equation:bar}
\int_{\partial_\infty \hypkn} d\beta_o|_{(\barb(\nu),\xi)}( \cdot )d\nu(\xi)=0 \ ,
\end{equation}
where $\nu \in \calM^1(\partial_\infty \hypkn)$ and $d\beta_o$ denotes the differential of the Busemann function pointed at $o \in \hypkn$. This property will be crucial to proving the smoothness of the slices of our natural map and to obtaining also the estimate on the Jacobian.

\end{description}

\subsection{Family of Patterson-Sullivan measures and BCG natural map} \label{subsec:patterson:sullivan}

In this section we are going to recall the definition of Patterson-Sullivan measures and the notion of natural map associated to a representation. We refer the reader to \cite{bcg95,bcg96,bcg98,franc06:articolo,franc09} for a more detailed exposition about these notions. 

Let $\upG(n)$ be equal to either $\po(n,1),\pu(n,1)$ or $\psp(n,1)$. It is well known that the Riemannian symmetric space associated to the rank one Lie group $\upG(n)$ is the hyperbolic space $\hypkn$ on a suitable division algebra $\upK$. More precisely we have that $\upK=\bbR$ if $\upG(n)=\po(n,1)$, $\upK=\bbC$ if $\upG(n)=\pu(n,1)$ and $\upK=\bbH$ if $\upG(n)=\psp(n,1)$. In all these cases, we are going to normalize the Riemannian metric on $\hypkn$ so that the sectional curvature has maximum value equal to $-1$. Denote by $d=\dim_{\bbR}\upK$ and assume that $d(n-1)\geq2$. 

Fix now a torsion-free (uniform) lattice $\Gamma \leq \upG(n)$. 

\begin{deft}\label{def:critical:exp}
Let $x \in \hypkn$ be any point and let $s >0$ be a real number. The \emph{s-Poincar\'e series pointed at $x$} is given by the following sum
$$
\calP(s;x):=\sum_{\gamma \in \Gamma} e^{-sd(\gamma x,x)} \ ,
$$
where $d$ stands for the distance induced by the fixed Riemannian structure on the space $\hypkn$. The \emph{critical exponent} associated to the lattice $\Gamma$ is defined as
$$
\delta_\Gamma:=\inf \{ s>0| \calP(s;x) < \infty\} \ . 
$$
The definition of critical exponent does not depend on the choice of the particular point $x \in \hypkn$ we fixed. 
\end{deft}

The critical exponent associated to a torsion-free (uniform) lattice in a rank-one Lie group is always finite and equal to
$$
\delta_\Gamma=d(n+1)-2 \ ,
$$
as shown for instance in \cite[Theorem 2]{albuquerque97}. We remind the reader that when $s=\delta_\Gamma$ the Poincar\'e series $\calP(s;x)$ diverges, that is $\calP(\delta_\Gamma;x)=+\infty$ \cite{sullivan79,burger:mozes,Yue96}, and for this reason we call $\Gamma$ a group of \emph{divergence type}. 

Now we are ready to give the definition of Patterson-Sullivan measures. This notion fits into a more general concept of conformal density.

\begin{deft}\label{def:patt:sull}
Let $\Gamma \leq \upG(n)$ be a torsion-free (uniform) lattice. Fix a positive real number $\alpha >0$. An \emph{$\alpha$-conformal density} for the lattice $\Gamma$ is a measurable map 
$$
\nu:\hypkn \rightarrow \calM^1(\partial_\infty \hypkn), \hspace{5pt} \nu(a):=\nu_a \ , 
$$
which satisfies the following conditions:
\begin{description}
\item[(i)] it is $\Gamma$-equivariant, that is $\nu_{\gamma a}=\gamma_\ast \nu_a$ for every $\gamma \in \Gamma$ and every $a \in \hypkn$. The symbol $\gamma_\ast$ stands for the push-forward measure with respect $\gamma$. 
\item[(ii)] given two different points $a,b \in \hypkn$, the measure $\nu_a$ is absolutely continuous with respect to $\nu_b$ and the Radon-Nikodym derivative is given by 
$$
\frac{d\nu_a}{d\nu_b}(\xi)=e^{-\alpha \beta_b(a,\xi)} \ ,
$$
where $\xi \in \partial_\infty \hypkn$ and $\beta_b(a,\xi)$ is the Busemann function pointed at $b$.  
\end{description}
When $\alpha$ is equal to the critical exponent $\delta_\Gamma$, the $\delta_\Gamma$-conformal density associated to $\Gamma$ is called the \emph{Patterson-Sullivan density}.
\end{deft}

Given a lattice $\Gamma \leq \upG(n)$, there always exists a Patterson-Sullivan density associated to it. Moreover it is essentially unique by the doubly ergodic action of $\Gamma$ on the boundary at infinity $\partial_\infty \hypkn$ (see for instance \cite{sullivan79,nicholls,burger:mozes,roblin,franc09}). It is worth mentioning that the construction of the Patterson-Sullivan density has been extended by Albuquerque \cite{albuquerque97,albuquerque99} to higher rank lattices in a semisimple Lie group $G$ of non-compact type. In the higher rank case the support of the measures is strictly smaller than the boundary at infinity of symmetric space $\calX_G$ associated to $G$. Indeed the support can be identified with the Furstenberg-Poisson boundary $B(G)$ (see Section \ref{subsec:zimmer:cocycle} for the definition). 

We conclude the section by recalling briefly the construction of the natural map associated to a representation. This can help the reader to understand how we are going to adapt the construction to the case of measurable cocycles. Let $\Gamma \leq \upG(n)$ be as above and consider $\rho:\Gamma \rightarrow \upG(m)$ a non-elementary representation, with $d(n-1)\geq2$ and $m \geq n$. Denote by $\{ \nu_a \}_{a \in \hypkn}$ the Patterson-Sullivan density of measure (actually its image).

Since the representation $\rho$ is non-elementary, by \cite[Corollary 3.2]{burger:mozes} there exists a measurable $\rho$-equivariant map 
$$
\varphi:\partial_\infty \hypkn \rightarrow \partial_\infty \hypkm \ .
$$
Additionally this map is essentially injective by both \cite[Lemma 2.3]{savini:articolo} and \cite[Lemma 2.5]{savini2:articolo}. This implies that for almost every $a \in \hypkn$ the push-forward measure $\varphi_\ast(\nu_a)$ has no atom. This condition allows us to define the Besson-Courtois-Gallot natural map associated to $\rho$.

\begin{deft}\label{def:natural:map:rep}
Let $\Gamma \leq \upG(n)$ be a (uniform) lattice and let $\rho:\Gamma \rightarrow \upG(m)$ be a non-elementary representation, with $d(n-1)\geq 2$ and $m \geq n$. If $\varphi:\partial_\infty \hypkn \rightarrow \partial_\infty \hypkm$ is the associated measurable map, we define the \emph{natural map associated to $\rho$} as follows
$$
F: \hypkn \rightarrow \hypkm, \hspace{5pt} F(a):=\barb(\varphi_\ast(\nu_a)) \ . 
$$
\end{deft}

The map defined above is smooth and $\rho$-equivariant, that is $F(\gamma a)=\rho(\gamma)F(a)$ for every $a \in \hypkn$. For every positive integer $p \in \bbN$, we define the \emph{$p$-Jacobian of $F$ at a} as
$$
\jac^p_a F:= \max_{u_1,\ldots,u_p \in T_a \hypkn} \lVert D_aF(u_1) \wedge \cdots \wedge D_aF(u_p) \rVert_{m} \ ,
$$
where $\{ u_1,\ldots,u_p\}$ is an orthonormal $p$-frame on the tangent space $T_a\hypkn$ with respect to the standard Riemannian metric $g_n$ and $\lVert  \ \cdot \ \rVert_m$ stands for the norm induced by $g_m$. When $p=n \cdot d$, that is, it is equal to the real dimension of $\hypkn$, we are going to denote the $p$-Jacobian simply by $\jac_aF$. For the natural map $F:\hypkn \rightarrow \hypkm$ it holds 
$$
\jac_aF \leq 1 \ ,
$$
for every $a \in \hypkn$ and the equality is attained if and only if the map $D_aF:T_a \hypkn \rightarrow T_{F(a)}\hypkm$ is an isometric embedding (see \cite[Lemma 7.2]{bcg95} for a proof of the inequality and \cite[Appendix B]{bcg95} for the study of the equality case). 

Since the natural map is defined using the barycenter and the latter satisfies Equation (\ref{equation:bar}), in this context one can verify that it holds
\begin{equation}\label{equation:natural:map}
\int_{\partial_\infty \hypkn} d\beta_o|_{(F(a),\varphi(\xi))} ( \  \cdot \  )d\nu_a(\xi)=0 \ .
\end{equation}
By differentiating the previous equation, for every $a \in \hypkn, u \in T_a \hypkn, v \in T_{F(a)} \hypkm$ one obtains that
\begin{align}
&\int_{\partial_\infty \hypkn} \nabla d\beta_o|_{(F(a),\varphi(\xi))} (D_aF(u),v)d\nu_a(\xi)=\\
&=\delta_\Gamma \int_{\partial_\infty \hypkn} d\beta_o|_{(F(a),\varphi(\xi))} (v) d\beta_o|_{(a,\xi)}(u)d\nu_a(\xi) \nonumber \ . 
\end{align}
Here $\nabla$ is the Levi-Civita connection associated to the natural Riemannian metric on $\hypkm$. We warn the reader that the Busemann functions that appear in the second line of the equation above refer to hyperbolic spaces of different dimensions (the first is defined on $\hypkm$ and the second one on $\hypkn$, respectively). 

\section{Natural maps associated to Zimmer's cocycles} \label{sec:natural:map}

In this section we are going to define the natural map associated to a measurable cocycle of a \emph{uniform} hyperbolic lattice. The key point in the construction will be to consider measurable cocycles which admit a boundary map whose slices are \emph{atomless}, that is the push-forward of the Patterson-Sullivan measure has no atom. This assumption will allow us to mimic the techniques used for non-elementary representations. 

\begin{proof}[Proof of Theorem \ref{teor:natural:map}]
Given $x \in X$, consider
$$
\phi_x:\partial_\infty \hypkn \rightarrow \partial_\infty \hypkm \ ,
$$
the slice of the boundary map $\phi$. For any $a \in \hypkn$, let $\nu_a$ be the Patterson-Sullivan measure pointed at $a$. If we consider the push-forward measure $(\phi_x)_*(\nu_a)$ this has no atoms by assumption. Hence we can apply the barycenter to define our desired map. More precisely, define
\begin{equation}\label{equation:natural:map:cocycle}
F:\hypkn \times X \rightarrow \hypkm, \hspace{5pt} F(a,x):=\barb((\phi_x)_\ast(\nu_a)) \ , 
\end{equation}
for every $a \in \hypkn$ and almost every $x \in X$. 

Clearly the map $F$ is well-defined by what we have said so far. Now we prove the $\sigma$-equivariance. Notice first that for almost every $x \in X$ we have a map $F_x:\hypkn \rightarrow \hypkm$ given by $F_x(a):=F(a,x)$. We call this map the \emph{$x$-slice of the map $F$}. Additionally we can suppose that the full-measure subset of $X$ on which the slices of $F$ are defined is $\Gamma$-invariant. Hence, given $\gamma \in \Gamma$, it holds
\begin{align*}
F(\gamma a,\gamma x)=&\barb((\phi_{\gamma x})_\ast(\nu_{\gamma a}))=\\
=&\barb((\phi_{\gamma x})_\ast(\gamma_\ast \nu_a))=\\
=&\barb(\sigma(\gamma,x)_\ast ((\phi_x)_\ast(\nu_a)))=\\
=&\sigma(\gamma,x)\barb((\phi_x)_\ast(\nu_a))=\sigma(\gamma,x)F(a,x) \ , 
\end{align*}
for every $a \in \hypkn$ and almost every $x \in X$. To pass from the first line to the second one we used the $\Gamma$-equivariance of the Patterson-Sullivan density and to move from the second line to the third one we exploited Equation (\ref{equation:gamma:slice}). To conclude we used the fact that the barycenter is $\upG(m)$-equivariant. The previous computation shows the $\sigma$-equivariance of the map $F$, as desired. 

Since the barycenter is characterized by the implicit Equation (\ref{equation:bar}), for almost every $x \in X$ we have that the slice $F_x$ satisfies the equation
\begin{equation}\label{equation:slice:implicit}
\int_{\partial_\infty \hypkn} d\beta_o|_{(F_x(a),\phi_x(\xi))}(\cdot)d\nu_a(\xi)=0 \ , 
\end{equation}
for every $a \in \hypkn$. By differentiating with respect to the variable $a$ the previous equation, for almost every $x \in X$ and every $a \in \hypkn, u \in T_a\hypkn,v \in T_{F_x(a)}\hypkm$ we obtain that 
\begin{align}\label{equation:diff:slice}
&\int_{\partial_\infty \hypkn} \nabla d\beta_o|_{(F_x(a),\phi_x(\xi))} (D_aF_x(u),v)d\nu_a(\xi)=\\
&=\delta_\Gamma \int_{\partial_\infty \hypkn} d\beta_o|_{(F_x(a),\phi_x(\xi))} (v) d\beta_o|_{(a,\xi)}(u)d\nu_a(\xi) \nonumber \ , 
\end{align}
where $\nabla$ is the Levi-Civita connection on $\hypkm$. Applying the same reasoning of Besson, Courtois and Gallot exposed in \cite{bcg95,bcg96} to the equation above one gets that the slice $F_x:\hypkn \rightarrow \hypkm$ is smooth for almost every $x \in X$, as claimed. 

We want to conclude by proving the estimate on the Jacobian of $F_x$. Fix an $x$ for which the slice $F_x$ is smooth. Recall that, for every $a \in \hypkn$, we are allowed to define three quadratic forms, one on the tangent space $T_a\hypkn$ and two on the tangent space $T_{F(a)}\hypkm$, respectively. More precisely, given $u \in T_a\hypkn$ and $v \in T_{F_x(a)}\hypkm$ we define 
\begin{align*}
h'_{a,x}(u,u)&:=\langle H'_{a,x}(u),u \rangle_n=\int_{\partial_\infty \hypkn} \left(d\beta_o|_{(a,\xi)}(u)\right)^2d\nu_a(\xi) \ , \\
h_{a,x}(v,v)&:=\langle H_{a,x}(v),v \rangle_m = \int_{\partial_\infty \hypkn} \left(d\beta_o|_{(F_x(a),\phi_x(\xi))}\right)^2d\nu_a(\xi) \ , \\
k_{a,x}(v,v)&:=\langle K_{a,x}(v),v \rangle_m = \int_{\partial_\infty \hypkn} \nabla d\beta_o|_{(F_x(a),\phi_x(\xi))}(v,v)d\nu_a(\xi) \ . 
\end{align*}
Here $H'_{a,x},H_{a,x}$ and $K_{a,x}$ are the endomorphisms associated to the symmetric bilinear forms with respect to the scalar product $\langle \cdot, \cdot \rangle_n$ (respectively $\langle \cdot,\cdot \rangle_m$) associated to the natural Riemannian metric on $\hypkn$ (respectively $\hypkm$). It is worth noticing that both $h_{a,x}$ and $h'_{a,x}$ are positive semidefinite bilinear forms, thus we are allowed to consider their square roots. As a consequence, applying the Cauchy-Schwarz inequality as in \cite[Section 2]{bcg98} to Equation (\ref{equation:diff:slice}), we get
\begin{equation}\label{equation:bilinear:forms}
k_{a,x}(D_aF_x(u),v) \leq \delta_\Gamma \left(h_{a,x}(v,v) \right)^{\frac{1}{2}}(h'_{a,x}(u,u))^{\frac{1}{2}} \ .
\end{equation}
Let $V_{a,x}:=D_aF_x(T_a\hypkn)$ be the image of the tangent space through the derivative at $a$ of the slice $F_x$. Denote by $h^V_{a,x},k^V_{a,x}$ the restrictions of the bilinears forms $h_{a,x},k_{a,x}$ to the subspace $V$. Let $H^V_{a,x},K^V_{a,x}$ be the associated endomorphisms. Set $p=d \cdot n$. By taking the determinant of Equation (\ref{equation:bilinear:forms}) one gets the following sequence of estimates
\begin{align*}
\det(K^V_{a,x})\jac_aF_x &\leq (\delta_\Gamma)^{p}\left( \det H^V_{a,x} \right)^{\frac{1}{2}} \left( \det H_{a,x}' \right)^{\frac{1}{2}} \leq \\
&\leq (\delta_\Gamma)^{p} \left( \det H^{V}_{a,x} \right)^{\frac{1}{2}} \left( \tr H'_{a,x}/p\right)^{\frac{p}{2}} \leq \\
&\leq p^{-\frac{p}{2}}\cdot (\delta_\Gamma)^{p} \left( \det H^V_{a,x} \right)^{\frac{1}{2}}.
\end{align*}
It is worth noticing that the estimate above depends on a suitable choice of basis of both $T_a\hypkn$ and $V_{a,x}$. We refer the reader to \cite[Lemma 5.3]{bcg98} for more details.

Finally, by applying \cite[Proposition B.1]{bcg95}, we obtain that
$$
\jac_aF_x \leq \frac{ (\delta_\Gamma)^{p} }{p^{-\frac{p}{2}}} \frac{ \left(\det H^V_{a,x} \right)^{\frac{1}{2}}}{\det(K^V_{a,x})} \leq 1 \ ,
$$
and the desired estimate is proved. When the equality is attained, the proof that $D_aF_x:T_a \hypkn \rightarrow T_{F_x(a)}\hypkm$ is an isometric embedding is analogous to the one exposed in \cite{bcg95,bcg96,bcg98} and hence we omit it. 
\end{proof}

\begin{oss}\label{oss:uniform:volume}
Notice that in the proof of Theorem \ref{teor:natural:map} we never use the uniformity of the lattice $\Gamma$. This implies that the same construction can be suitably extended also to torsion-free \emph{non-uniform} lattices. In the latter case one would like to prove a property similar to the properly ending condition defined in \cite{franc06:articolo,franc09}. However we cannot understand a clear way to do it in this context. 
\end{oss}

\begin{oss}
One could ask when a boundary map has atomless slices and how restrictive this hypothesis is. When $\sigma$ is the measurable cocycle associated to a self-coupling of a uniform lattice $\Gamma \leq \upG(n)$, the slices of the boundary map are atomless by either \cite[Lemma 3.6]{sauer:articolo} or \cite[Lemma 3.1]{savini4}. More generally if $(X,\mu_X)$ is ergodic and the cocycle is Zariski dense, this is sufficient to guarantee atomless slices. 
\end{oss}

Recall that in Section \ref{subsec:zimmer:cocycle} we discussed how to construct a suitable measurable cocycle starting from a representation $\rho:\Gamma \rightarrow \upG(m)$ once we fixed a standard Borel probability $\Gamma$-space. One could naturally ask whether there exists a relation between the natural map associated to $\rho$ defined in \cite{franc06:articolo,franc09} and the natural map we defined in Theorem \ref{teor:natural:map}. Their link is given by the following:

\begin{prop}\label{prop:natural:rep}
Let $\Gamma \leq \upG(n)$ be a torsion-free uniform lattice and let $\rho:\Gamma \rightarrow \upG(m)$ be a non-elementary representation, with $m \geq n$. Fix a standard Borel probability $\Gamma$-space $(X,\mu_X)$. Denote by $\widetilde{F}:\hypkn \rightarrow \hypkm$ and by $\sigma_\rho:\Gamma \times X \rightarrow \upG(m)$ the natural map and the measurable cocycle associated to $\rho$, respectively. Then the natural map associated to $\sigma_\rho$ is given by 
$$
F:\hypkn \times X \rightarrow \hypkm, \hspace{5pt} F(a,x):=\widetilde{F}(a) \ ,
$$
for every $a \in \hypkn$ and almost every $x \in X$.
\end{prop}

\begin{proof}
Since $\rho$ is non-elementary, by \cite{burger:mozes,franc09} there exists a measurable boundary map $\widetilde{\varphi}:\partial_\infty \hypkn \rightarrow \partial_\infty \hypkm$ which is $\rho$-equivariant. Additionally, since the $\Gamma$ action on the boundary $\partial_\infty \hypkn$ is doubly ergodic, the map is essentially unique (\cite{sullivan79,nicholls,burger:mozes,roblin,franc09}) and essentially injective (\cite[Lemma 2.3]{savini:articolo},\cite[Lemma 2.5]{savini2:articolo}). 

The previous map allows to define an essentially unique boundary map associated to $\sigma_\rho$ as follows
$$
\phi:\partial_\infty \hypkn \times X \rightarrow \partial_\infty \hypkm, \hspace{5pt} \phi(\xi,x):=\widetilde{\varphi}(\xi) \ ,
$$
for almost every $\xi \in X$ and every $x \in X$. Moreover, since every slice $\phi_x$ coincides with the map $\widetilde{\varphi}$, every slice is essentially injective. Hence we are in the hypothesis of Theorem \ref{teor:natural:map} and we are allowed to construct the natural map associated to $\sigma_\rho$. By definition we have that
$$
F(a,x)=\barb((\phi_x)_\ast(\nu_a))=\barb((\widetilde{\varphi})_\ast(\nu_a))=\widetilde{F}(a) \ , 
$$
and the claim is proved. 
\end{proof}

We conclude the section by showing how natural maps change in the $\upG(m)$-cohomology class of a given measurable cocycle. 

\begin{prop}\label{prop:natural:cohomology}
Let $\Gamma \leq \upG(n)$ be a torsion-free uniform lattice and fix $(X,\mu_X)$ a standard Borel probability $\Gamma$-space. Let $\sigma:\Gamma \times X \rightarrow \upG(m)$ be a measurable cocycle which admits a natural map $F:\hypkn \times X \rightarrow \hypkm$. Given a measurable map $f:X \rightarrow \upG(m)$, the natural map associated to the cocycle $\sigma^f$ is given by 
$$
F^f:\hypkn \times X \rightarrow \hypkm, \hspace{5pt} F^f(a,x)=f(x)^{-1}F(a,x) \ . 
$$
\end{prop}

\begin{proof}
Denote by $\phi:\partial_\infty \hypkn \times X \rightarrow \partial_\infty \hypkm$ the boundary map associated to $\sigma$. Recall by Definitions \ref{def:coom:cocycle} and \ref{def:twist:bound:map} that the twisted cocycle
$$
\sigma^f:\Gamma \times X \rightarrow \upG(m), \hspace{5pt} \sigma^f(\gamma,x)=f(\gamma x)^{-1}\sigma(\gamma,x)f(x) \ ,
$$
admits as boundary map
$$
\phi^f:\partial_\infty \hypkn \times X \rightarrow \partial_\infty \hypkm, \hspace{5pt} \phi^f(\xi,x)=f(x)^{-1}\phi(\xi,x) \ .
$$
Notice that $f$ takes values into $\upG(m)$, hence if almost every slice of $\phi$ is atomless, the same holds for $\phi^f$. By definition of the associated natural map we have
$$
F^f(a,x)=\barb((\phi^f_x)_\ast(\nu_a))=\barb((f^{-1}(x)\phi_x)_\ast(\nu_a)) \ .
$$
Since the barycenter is $\upG(m)$-equivariant we obtain
$$
\barb((f^{-1}(x)\phi_x)_\ast(\nu_a))=f(x)^{-1}\barb((\phi_x)_\ast(\nu_a))=f(x)^{-1}F(a,x) \ ,
$$
and the statement follows. 
\end{proof}

\section{Natural volume of Zimmer's cocycles} \label{sec:vol:coc} 

In this section we are going to introduce the notion of natural volume of a measurable cocycle of a \emph{uniform} hyperbolic lattice. As already discussed in the introduction, even if we are going to focus our attention only on uniform lattices, the same results of the section will hold also in the non-uniform case. Notice also that the definition we are going to give differs from the one given in \cite{moraschini:savini,moraschini:savini:2} since ours relies on the differentiability of the slices of \emph{essentially bounded} equivariant maps (see Definition \ref{def:volume:map}). Moreover the rigidity result we are going to obtain will refer to cocycles associated to lattices of $\upG(n)$ with image into $\upG(m)$, with $m$ possibly greater than or equal to $n$. 

Let $\Gamma \leq \upG(n)$ be a torsion-free uniform lattice and fix $(X,\mu_X)$ a standard Borel probability $\Gamma$-space. Denote by $\hypkn$ the hyperbolic space over the division algebra $\upK$ associated to $\upG(n)$. Set $d=\dim_{\bbR}\upK$ and assume $d(n-1) \geq 2$. Take $m \geq n$ and consider a measurable cocycle $\sigma:\Gamma \times X \rightarrow \upG(m)$ which admits a measurable $\sigma$-equivariant map $\Phi:\hypkn \times X \rightarrow \hypkm$ whose slice $\Phi_x:\hypkn \rightarrow \hypkm, \hspace{5pt} \Phi_x(a):=\Phi(a,x)$ is differentiable, for almost every $x \in X$. This implies that we can consider the determinant associated to the pullback of the metric $g_m$ with respect to the slice $\Phi_x$. More precisely, if we set $p=d \cdot n$, we can define the $p$-form 
$$
(\omega_x)_a(u_1,\ldots,u_p):=\sqrt{\det((\Phi_x^\ast) g_m)}(u_1,\ldots,u_p)=\sqrt{ \det \langle D_a\Phi_x(u_i),D_a\Phi_x(u_j) \rangle_m}\ ,
$$
for almost every $x \in X$, every $a \in \hypkn$ and every $u_1,\ldots,u_p \in T_a\hypkn$. In this way we get a family $\{ \omega_x \}_{x \in X}$ of differential forms on $\hypkn$. For almost every $x \in X$ and every $a \in \hypkn$, we can exploit the Riemannian structure on $\hypkn$ to define the norm
$$
\lVert (\omega_x)_a \rVert_\infty:= \max_{u_1,\ldots,u_p \in T_a\hypkn} |(\omega_x)_a(u_1,\ldots,u_p)| \ ,
$$
where the set $\{ u_1,\ldots, u_p \}$ varies on the set of all the possible orthonormal frames of $T_a\hypkn$. We are going to say that $\Phi$ is \emph{essentially bounded} (or has \emph{essentially bounded slices}) if there exists a real number $C>0$ such that 
$$
\lVert (\omega_x)_a \rVert_\infty < C \ ,
$$
for almost every $x \in X$ and every $a \in \hypkn$. Notice that this condition is not so restrictive, since for instance the natural map associated to a measurable cocycle satisfies this property (see Remark \ref{oss:essential:bound}).

Assume now that $\Phi$ is essentially bounded. We are allowed to integrate with respect to the $x$-variable and in this way we get a well-defined differential form on $\hypkn$, that is
$$
\widetilde{\omega}_X:=\int_X \omega_x d\mu_X(x) \in \Omega^p(\hypkn; \bbR) \ , 
$$
$$
\widetilde{\omega}_X(u_1,\ldots,u_p)=\int_X\omega_x(u_1,\ldots,u_p)d\mu_X(x) \ .
$$
We claim that $\widetilde{\omega}_X$ is $\Gamma$-invariant, that is $\widetilde{\omega}_X \in \Omega^p(\hypkn;\bbR)^\Gamma$ and hence it induces a well-defined differential form $\omega_X \in \Omega^p(\Gamma \backslash \hypkn;\bbR)$. More precisely, let $\gamma \in \Gamma$. We need to show that 
$$
\gamma^\ast \widetilde{\omega}_X=\widetilde{\omega}_X \ ,
$$
or equivalently 
$$
\widetilde{\omega}_X(D_a\gamma(u_1),\ldots,D_a\gamma(u_p))=\widetilde{\omega}_X(u_1,\ldots,u_p) \ ,
$$
for every $a \in \hypkn$ and $u_1,\ldots,u_p \in T_a\hypkn$. It holds

\begin{align*}
\gamma^\ast \widetilde{\omega}_X(u_1,\ldots,u_p)&=\omega_X(D_a\gamma(u_1),\ldots,D_a\gamma(u_p))=\\
&=\int_X \omega_x(D_a\gamma (u_1),\ldots,D_a\gamma (u_p))d\mu_X(x)=\\
&=\int_X \omega_{\gamma y}(D_a\gamma (u_1),\ldots,D_a\gamma (u_p))d\mu_X(y) =(\bullet ) \ ,
\end{align*}
where we set $x=\gamma y$ and we used the fact that $\Gamma$ acts on $X$ via measure preserving transformations. From the formula above we can argue 
\begin{align*}
(\bullet)&=\int_X  \sqrt{\det \left( \langle  D_{\gamma a}\Phi_{\gamma y}(D_a\gamma(u_i)) ,D_{\gamma a}\Phi_{\gamma y}(D_a\gamma(u_j)) \rangle_m  \right)}d\mu_X(y)=\\
&=\int_X  \sqrt{\det \left( \langle  D_{a}(\Phi_{\gamma y}\circ \gamma)(u_i)) ,D_{a}(\Phi_{\gamma y}\circ \gamma)(u_j)) \rangle_m  \right)}d\mu_X(y)=\\
&=\int_X  \sqrt{\det \left( \langle  D_{a}(\sigma(\gamma,y)\Phi_y)(u_i)) ,D_{a}(\sigma(\gamma,y)\Phi_y)(u_j)) \rangle_m  \right)}d\mu_X(y)=\\
&=\int_X  \sqrt{\det \left( \langle  D_{a}(\Phi_y)(u_i)) ,D_{a}(\Phi_y)(u_j)) \rangle_m  \right)}d\mu_X(y)=\widetilde{\omega}_X(u_1,\ldots,u_p) \ , 
\end{align*}
where we used the $\sigma$-equivariance of $\Phi$ to pass from the second line to the third one and we exploited the fact that the cocycle $\sigma$ takes value into the isometry group $\upG(m)$ to conclude. Hence we get that $\widetilde{\omega}_X \in \Omega^p(\hypkn;\bbR)^\Gamma$ and so we obtain a well-defined differential form $\omega_X \in \Omega^p(\Gamma \backslash \hypkn;\bbR)$.

\begin{deft}\label{def:volume:map}
Let $\Gamma \leq \upG(n)$ and let $(X,\mu_X)$ be a standard Borel probability $\Gamma$-space. Let $\sigma:\Gamma \times X \rightarrow \upG(m)$ be a measurable cocycle, with $m \geq n$. Denote by $\mathscr{D}(\sigma)$ the set of essentially bounded $\sigma$-equivariant maps with differentiable slices. Given $\Phi \in \mathscr{D}(\sigma)$, we define the \emph{volume associated to the map $\Phi$} as 
$$
\vol(\Phi):=\int_{\Gamma \backslash \hypkn} \omega_X=\int_{\Gamma \backslash \hypkn} \int_X \omega_xd\mu_X(x)=\int_{\Gamma \backslash \hypkn} \int_X \sqrt{\det(\Phi^\ast_x)g_m} d \mu_X(x) \ . 
$$
\end{deft}

\begin{oss}\label{oss:essential:bound}
Notice that when a measurable cocycle $\sigma:\Gamma \times X \rightarrow \upG(m)$ admits a natural map $F:\hypkn \times X \rightarrow \hypkm$, the latter is an essentially bounded $\sigma$-equivariant map with differentiable slices. The essential boundedness comes from the estimate on that Jacobian of the slices of $F$. Indeed, the family of differential forms associated to $F$ can be written as
$$
\omega_x:=\sqrt{\det F^\ast_x g_m}=\jac F_x \cdot \omega_n \ , 
$$
where $\omega_n$ is the standard volume form on $\hypkn$. In particular, if we fix an orthonormal frame $\{u_1,\ldots,u_p\}$ of the tangent space $T_a \hypkn$ we have that
$$
|(\omega_x)_a(u_1,\ldots,u_p)| = \jac_a F_x \cdot |(\omega_n)_a(u_1,\ldots,u_p)| \leq 1 \ ,
$$
and this implies that $F$ is essentially bounded. Hence the set $\mathscr{D}(\sigma)$ of essentially bounded $\sigma$-equivariant map with differentiable slices is not empty, since it contains at least $F$. 
\end{oss}

The previous Remark allows us to define the notion of natural volume of a measurable cocycle.

\begin{deft}\label{def:volume:coc}
Let $\Gamma \leq \upG(n)$ and let $(X,\mu_X)$ be a standard Borel probability $\Gamma$-space. Let $\sigma:\Gamma \times X \rightarrow \upG(m)$ be a measurable cocycle, with $m \geq n$. Assume that $\sigma$ admits a natural map $F: \hypkn \times X \rightarrow \hypkm$. The \emph{natural volume of the cocycle $\sigma$} is defined as
$$
\nv(\sigma):=\vol(F) \ .
$$
\end{deft}

\begin{oss}
This definition might seem quite strange to the reader who is confident with the work of Francaviglia and Klaff \cite{franc06:articolo,franc09}. If we wanted to follow their theoretical line we should have considered the infimum of the volumes over all the possible elements of $\mathscr{D}(\sigma)$. However in that case, we would not be able to prove that the volume of the standard lattice embedding is maximal. 
\end{oss}

The first thing we want to show is that the notion of volume we gave is actually invariant along the $\upG(m)$-cohomology class of a fixed measurable cocycle, in analogy to what happens to the volume defined in \cite{moraschini:savini}. 

\begin{prop}\label{prop:volume:cohomology}
Let $\Gamma \leq \upG(n)$ be a torsion-free uniform lattice and let $(X,\mu_X)$ be a standard Borel probability $\Gamma$-space. Consider $\sigma:\Gamma \times X \rightarrow \upG(m)$ a measurable cocycle which admits a natural map. Given any measurable map $f:X \rightarrow \upG(m)$ it holds
$$
\nv(\sigma^f)=\nv(\sigma) \ ,
$$
and hence the natural volume in constant along the $\upG(m)$-cohomology class of $\sigma$. 
\end{prop}

\begin{proof}
Recall that $\mathscr{D}(\sigma)$ (respectively $\mathscr{D}(\sigma^f)$) is the set of all the possible essentially bounded $\sigma$-equivariant (respectively $\sigma^f$-equivariant) maps with differentiable slices. It is easy to verify that given $\Phi \in \mathscr{D}(\sigma)$ we can define 
$$
\Phi^f:\hypkn \times X \rightarrow \hypkm, \hspace{5pt} \Phi^f(a,x):=f(x)^{-1}\Phi(a,x) \ ,
$$
and clearly $\Phi^f \in \mathscr{D}(\sigma^f)$. In this way we obtain a bijection
$$
{ }^f:\mathscr{D}(\sigma) \rightarrow \mathscr{D}(\sigma^f), \hspace{5pt} \Phi \mapsto \Phi^f \ . 
$$

We are going  to prove that this bijection preserves the volume, that is $\vol(\Phi)=\vol(\Phi^f)$ for every $\Phi \in \mathscr{D}(\sigma)$. Indeed we have that
\begin{align*}
\vol(\Phi^f)&=\int_{\Gamma \backslash \hypkn} \int_X \sqrt{\det ( \Phi^f_x )^\ast g_m} d\mu_X(x)=\\
&=\int_{\Gamma \backslash \hypkn} \int_X \sqrt{\det (f(x)^{-1}\Phi_x)^\ast g_m}d\mu_X(x)=\\
&=\int_{\Gamma \backslash \hypkn} \int_X \sqrt{\det (\Phi_x)^\ast g_m}d\mu_X(x)=\vol(\Phi) \ ,
\end{align*}
where to pass from the second to the third line we used the fact that $f(x) \in \upG(m)$ is an isometry for $g_m$. 

If we now restrict our attention to the case of natural maps, we already proved in Proposition \ref{prop:natural:cohomology} that if $F$ is the natural map associated to $\sigma$, then $F^f$ is the one associated to $\sigma^f$. From this consideration it follows that
$$
\nv(\sigma^f)=\vol(F^f)=\vol(F)=\nv(\sigma) \ , 
$$
as desired. 

\end{proof}

After the study of the properties of the volume invariant of a given cocycle, we are now ready to prove our rigidity theorem, that is Theorem \ref{teor:rigidity:cocycle}. 

\begin{proof}[Proof of Theorem \ref{teor:rigidity:cocycle}]
We are going to show first the Milnor-Wood type inequality. Define $M=\Gamma \backslash \hypkn$. Consider $\sigma:\Gamma \times X \rightarrow \upG(m)$ and the associated set $\mathscr{D}(\sigma)$. As a consequence of Remark \ref{oss:essential:bound} we know that the natural map $F:\hypkn \times X \rightarrow \hypkm$ associated to $\sigma$ is an element of $\mathscr{D}(\sigma)$.

By the estimate on the Jacobian of the slices of $F$, it follows that
\begin{align*}
\nv(\sigma)&=\vol(F)= \\
&= \int_{M} \int_X \sqrt{ \det F^\ast_x g_m }d\mu_X(x)=\\
&=\int_{M} \left(\int_X \jac_a F_x d\mu_X(x)\right) \omega_M \leq\\
&\leq \int_{M} \left(\int_X d\mu_X(x)\right) \omega_M=\vol(M) \ ,
\end{align*}
where $\omega_M$ is the volume form associated to the standard Riemannian metric on $M$. In this way we get our desired inequality. 

We now prove the rigidity statement. We first introduce some notation. Denote by 
$$
i_{n,m}:\upG(n) \rightarrow \upG(m), \hspace{5pt} g \mapsto 
\left(
\begin{array}{cc}
g & 0 \\ 
0 & \upI_{m-n}\\
\end{array}
\right) \ ,
$$
the upper-left corner embedding. Here $\upI_{m-n}$ stands for the identity matrix of order $(m-n)$. Let $\jmath_{n,m}:\hypkn \rightarrow \hypkm$ be the totally geodesic embedding of $\hypkn$ into $\hypkm$ which is $i_{n,m}$-equivariant. 

Assume now that $\nv(\sigma)=\vol(M)$. By definition it follows that 
$$
\vol(M) = \nv(\sigma) = \vol(F) \ ,
$$
where $F$ still denotes the natural map associated to $\sigma$. By writing explicitly the equality above we get
\begin{equation}\label{equation:max:jac}
\vol(M) = \int_{M} \int_X \sqrt{ \det F^\ast_x g_m }d\mu_X(x) = \int_{M} \left(\int_X \jac_a F_x d\mu_X(x)\right) \omega_M \ . 
\end{equation}
Since $\jac_a F_x \leq 1$ for every $a \in \hypkn$ and almost every $x \in X$, Equation (\ref{equation:max:jac}) implies that 
$$
\jac_aF_x=1 \ ,
$$
for almost every point of the $\Gamma$-fundamental domain in $\hypkn$ and for almost every $x \in X$. By the $\sigma$-equivariance of $F$, we have that 
$$
\jac_aF_x=1 \ ,
$$
for almost every $a \in \hypkn$ and almost every $x \in X$. Now fix $x \in X$ and consider the slice $F_x:\hypkn \rightarrow \hypkm$ for which it holds $\jac_aF_x=1$ for almost every $a \in \hypkn$. By \cite{bcg95,franc06:articolo} it follows that $F_x$ coincides essentially with a totally geodesic embedding of $\hypkn$ into $\hypkm$. More precisely there must exists $f(x) \in \upG(m)$ such that 
$$
F_x(a)=f(x)\jmath_{n,m}(a) \ ,
$$
for almost every $a \in \hypkn$ (actually every $a \in \hypkn$ by the differentiability of the slices). In this way we obtain a map $f:X \rightarrow \upG(m)$. Since by assumption $X$ is a standard Borel space, the function $\widehat{F}:X \rightarrow \textup{Meas}(\hypkn,\hypkm), \hspace{5pt} \widehat{F}(x):=F_x$ is measurable by \cite[Lemma 2.6]{fisher:morris:whyte} and this implies the measurability of $f$. In this way we get a measurable map $f:X \rightarrow \upG(m)$ which conjugates $F$ to the totally geodesic embedding $\jmath_{n,m}$. Following the same strategy exposed in \cite[Proposition 3.2]{sauer:articolo} we claim that $\sigma$ is cohomologous to the restriction $i_{n,m}|_{\Gamma}$ of the upper-left corner embedding to the lattice $\Gamma$, modulo possibly a compact subgroup. 

Let $$C:=\textup{Stab}_{\upG(m)}(\jmath_{n,m})$$ be the subgroup of $\upG(m)$ fixing pointwise the image of $\jmath_{n,m}$. This is the trivial group when $n=m$ and it is compact when $m>n$.  Fix $\gamma \in \Gamma$. For almost every $a \in \hypkn,x \in X$, on one hand it holds
$$
F(\gamma a,\gamma x)=f(\gamma x)\jmath_{n,m}(\gamma a)=f(\gamma x) i_{n,m}(\gamma) \jmath_{n,m}(a) \ ,
$$
and on the other hand we have
$$
F(\gamma a, \gamma x)=\sigma(\gamma,x)F(a,x)=\sigma(\gamma,x)f(x)\jmath_{n,m}(a) \ .
$$
Hence it follows
$$
i_{n,m}(\gamma)=f(\gamma x)^{-1}\sigma(\gamma,x)f(x) \ \ \mod C 
$$
and the claim is proved. 

Viceversa, consider the cocycle $\sigma_{i_{n,m}}:\Gamma \times X \rightarrow \upG(m)$ associated to the representation $i_{n,m}$ restricted to $\Gamma$. It is easy to verify that the natural map associated to this cocycle is given by 
$$
F:\hypkn \times X \rightarrow \hypkm, \hspace{5pt} F(a,x):=\jmath_{n,m}(a) \ ,
$$
for every $a \in \hypkn$ and almost every $x \in \hypkn$. We have that
\begin{align*}
\nv(\sigma_{i_{n,m}})&=\vol(F)=\\
&=\int_{M} \left( \int_X \jac_aF_x d\mu_X(x) \right)\omega_M=\\
&=\int_{M} \left( \int_X \jac_a \jmath_{n,m} d\mu_X(x) \right) \omega_M=\\
&=\int_{M} \jac_a \jmath_{n,m} \omega_M=\int_M \omega_M=\vol(M) \ ,
\end{align*}
and the statement follows.
\end{proof}

It is worth noticing that all the results we have shown so far are still valid for non-uniform lattice, since we did not exploit the property of being uniform for the lattice $\Gamma$. 

We conclude the section by underling that importance of the previous theorem, since in the particular case when $n=m$, it has been exploited by the author \cite{savini4} to prove the $1$-tautness of the group $\pu(n,1)$, when $n \geq 2$. 

\section{Volume of equivariant maps and mapping degree}\label{sec:vol:degree} 

In this section we are going to show a suitable adaptation of the mapping degree theorem to the context of measurable cocycles associated to uniform lattices of rank-one Lie groups (see also \cite[Proposition 3.1]{derbez:liu:sun:wang}). We will introduce the notion of pullback of a measurable cocycle $\sigma$ with respect to a continuous map between closed manifolds. Since the same can be done for a measurable $\sigma$-equivariant map $\Phi$ with essentially bounded differentiable slices, we are going to show that the volume of the pullback $f^\ast \Phi$ bounds from above the volume of $\Phi$ multiplied by the mapping degree of the continuous map. 

Let $\Gamma,\Lambda \leq \upG(n)$ be two torsion-free uniform lattices. Denote by $M=\Gamma \backslash \hypkn$ and $N=\Lambda \backslash \hypkn$ the closed hyperbolic manifolds associated to $\Gamma$ and $\Lambda$, respectively. Consider a continuous map $f:M \rightarrow N$ and denote by $\pi_1(f):\Gamma \rightarrow \Lambda$ the homomorphism induced on the fundamental groups. Let $(X,\mu_X)$ be a standard Borel probability $\Lambda$-space. Following \cite[Section 6]{moraschini:savini}, given a measurable cocycle $\sigma:\Lambda \times X \rightarrow \upG(m)$ we define the \emph{pullback cocycle of $\sigma$ with respect to $f$} as follows
$$
f^\ast \sigma:\Gamma \times X \rightarrow \upG(m), \hspace{5pt} f^\ast \sigma(\gamma,x):=\sigma(\pi_1(f)(\gamma),x) \ .
$$
Here $(X,\mu_X)$ is viewed as a $\Gamma$-space with the action induced by the map $\pi_1(f)$. As shown in \cite[Lemma 6.1]{moraschini:savini} the map $f^\ast \sigma$ is a well-defined cocycle. 

Given a continuous map $f:M \rightarrow N$ with non-zero degree, it is well-known by \cite{burger:mozes,franc09} that there exists an essentially unique measurable map $\varphi:\partial_\infty \hypkn \rightarrow \partial_\infty \hypkn$ which is $\pi_1(f)$-equivariant. Hence, by following the approach of \cite{bcg95,bcg96,bcg98} there exists a natural map $\widetilde{F}:\hypkn \rightarrow \hypkn$ which is smooth and $\pi_1(f)$-equivariant. This map descends to a smooth map $F:M \rightarrow N$ which has the same degree of $f$. Additionally we recall the standard bound on the Jacobian of $\widetilde{F}$, that is $\jac_a \widetilde{F} \leq 1$ for every $a \in \hypkn$. 

Now consider a measurable $\sigma$-equivariant map $\Phi:\hypkn \times X \rightarrow \hypkm$. We can define the \emph{pullback of the map $\Phi$ along the continuous map $f$} as follows
\[
f^\ast \Phi: \hypkn \times X \rightarrow \hypkm, \hspace{5pt} f^\ast \Phi(a,x):=\Phi( \widetilde{F}(a),x) \ ,
\]
where $\widetilde{F}$ is the lift of the natural map previously described.
Having introduced all the notation we needed, we are now ready to prove our version of mapping degree theorem.

\begin{proof}[Proof of Proposition \ref{prop:map:deg}]
Up to changing the orientation to either $M$ or $N$ we can suppose that the degree of the map $f:M \rightarrow N$ is positive. Denote by $\omega_M$ and $\omega_N$ the volume forms induced on $M$ and on $N$ by the standard hyperbolic metric, respectively. 

Let now $\Phi:\hypkn \times X \rightarrow \hypkm$ be the equivariant map we have in statement. It holds the following chain of equalities 
 \begin{align*}
\vol(f^\ast \Phi)&=\int_{M} \int_X \sqrt{ \det (f^\ast \Phi)_x^\ast g_m } d\mu_X(x)=\\
&=\int_{M} \left( \int_X \jac_a(f^\ast \Phi)_x d\mu_X(x)\right) \omega_M =\\
&=\int_{M} \left( \int_X \jac_a(\Phi_x \circ \widetilde{F})d\mu_X(x) \right) \omega_M \ . \\
\end{align*}

We can now take out from the sign of integration along $X$ the Jacobian of the map $\widetilde{F}$, actually of the map $F$ by its equivariance property: 
\begin{align*}
\int_{M} \left( \int_X \jac_a(\Phi_x \circ \widetilde{F})d\mu_X(x) \right) \omega_M&=\int_{M} \jac_a F \left( \int_X \jac_{\widetilde{F}(a)} \Phi_x d\mu_X(x) \right) \omega_M = \\
&=\int_{N} \sum_{a \in F^{-1}(b)}\left( \int_X \jac_b \Phi_x d\mu_X(x) \right) \omega_N =\\
&= \int_{N} \calN(b) \left( \int_X \jac_b \Phi_x d\mu_X(x) \right) \omega_N \geq \\
&\geq \int_{N} \deg(F) \left( \int_X \jac_b \Phi_x d\mu_X(x) \right) \omega_N = \deg(f) \cdot \vol(\Phi) \ . 
\end{align*}
In the computation above, for any $b \in N$ we defined the number $\calN(b)$ as the cardinality
$$
\calN(b):=\textup{card}( \widetilde{F}^{-1}(b)) \ ,
$$
and we used the co-area formula to move from the first line to the second one. Since $\calN(b) \geq \deg(F)$ and $F$ has the same degree of $f$, the desired inequality follows.

Suppose now that $f$ is homotopic to a local isometry. By the strict version of mapping degree theorem (see \cite[Theorem 6.4]{Thurston} for the real hyperbolic case and \cite[Th\'eor\`em Principal]{bcg95} for the rank-one case) it follows that 
$$
\vol(M)=\deg(f) \cdot \vol(N)= \deg(F) \cdot \vol(N) \ ,
$$
since $F$ and $f$ have the same degree. We claim that this implies that $F:M \rightarrow N$ is a local isometry. Indeed following the same reasoning exposed in \cite[Appendix C]{bcg95} we have that
$$
\vol(M) \geq \int_M \jac_aF \cdot \omega_M = \int_N \calN(b) \cdot \omega_N \geq \deg(F) \cdot \vol(N) \ .
$$
In our particular case all the inequalities above are actually equalities and hence $\calN(b)=\deg(F)$ and, by the equivariance of $\widetilde{F}$, we obtain $\jac_a \widetilde{F}=1$ for almost every $a \in \hypkn$. This means that $\widetilde{F}: \hypkn \rightarrow \hypkn$ coincides essentially with an isometry (hence it is an isometry being differentiable). Thus the induced map $F:M \rightarrow N$ is a local isometry.  

By the conditions on both the cardinality $\calN(b)$ and on the Jacobian of $F$, if we substitute their respective values into the chain of equalities at the beginning of the proof, we get that
$$
\vol(f^\ast \Phi)=\deg(f) \cdot \vol(\Phi) \ , 
$$
as claimed. 

\end{proof}

We would like to apply the proposition above to argue a relation between the degree of the continuous map $f$ and the natural volumes of the cocycles $\sigma$ and $f^\ast \sigma$, respectively. Unfortunately, if $\Psi$ is the natural map associated to $\sigma$, it is not true a priori that $f^\ast \Psi$ is the natural map associated to $f^\ast \sigma$. For this reason a weaker statement given by the following

\begin{cor}\label{cor:natural:deg}
Let $\Gamma,\Lambda \leq \upG(n)$ be two torsion-free uniform lattices. Set $M=\Gamma \backslash \hypkn$, $N=\Lambda \backslash \hypkn$ and let $f:M \rightarrow N$ be a continuous map with non-vanishing degree. Fix a standard Borel probability $\Lambda$-space $(X,\mu_X)$ and consider a measurable cocycle $\sigma:\Lambda \times X \rightarrow \upG(m)$. Suppose that $\sigma$ admits a natural map $\Psi$. Then it holds 
$$
|\deg(f)| \leq \frac{\vol(f^\ast\Psi)}{\nv(\sigma)} \ .
$$
Additionally if $f$ is homotopic to a local isometry then the equality is attained. 
\end{cor}

\begin{proof}
The claim follows immediately by substituting $\Psi$ in Proposition \ref{prop:map:deg} and noticing that we have $\nv(\sigma)=\vol(\Psi)$ by definition. 
\end{proof}

\begin{oss}
In the situation of Corollary \ref{cor:natural:deg}, it is worth noticing that $\vol(f^\ast \Psi)$ is bounded from above by $\vol(M)$. Indeed we have that
\begin{align}\label{eq:pullback:bound}
\vol(f^\ast \Psi) &= \int_M \jac_a F \left( \int_X \jac_{\widetilde{F}(a)}\Psi_x d\mu_X(x) \right) \omega_M \leq \\
&\leq  \int_M \left( \int_X d\mu_X(x) \right) \omega_M =\vol(M) \nonumber \  , 
\end{align}
where we moved from the first line to the second one using the bound on the Jacobian of both natural maps $\widetilde F$ and $\Psi$. 
\end{oss}

We conclude this section by giving a characterization of continuous maps homotopic to local isometries in terms of naturally maximal cocycles. A cocycle $\sigma:\Lambda \times X \rightarrow \upG(m)$ is \emph{naturally maximal} if it admits a natural map $\Psi:\hypkn \times X \rightarrow \hypkm$ and it holds that
$$
\nv(\sigma)=\vol(\Lambda \backslash \hypkn) \ . 
$$
Compare the following result with \cite[Proposition 1.4]{moraschini:savini}.

\begin{cor}
Let $\Gamma,\Lambda \leq \upG(n)$ be two torsion-free uniform lattices. Set $M=\Gamma \backslash \hypkn$, $N=\Lambda \backslash \hypkn$ and let $f:M \rightarrow N$ be a continuous map with non-vanishing degree. Fix a standard Borel probability $\Lambda$-space $(X,\mu_X)$ and consider a naturally maximal cocycle $\sigma:\Lambda \times X \rightarrow \upG(m)$ with natural map $\Psi$. Then $f$ is homotopic to a local isometry if and only if $\vol(f^\ast \Psi)=\vol(M)$. 
\end{cor}

\begin{proof}
We are going to keep the same notation of both Proposition \ref{prop:map:deg} and Corollary \ref{cor:natural:deg}. Suppose that $f$ is homotopic to a local isometry. By Corollary \ref{cor:natural:deg} it follows 
$$
\vol(f^\ast \Psi)=\deg(f) \cdot \nv(\sigma)=\deg(f) \cdot \vol(N) \ ,
$$
and the latter equality is justified by the maximality assumption. As a consequence of the strict version of mapping degree theorem \cite[Theorem 6.4]{Thurston},\cite[Th\'eor\`em Principal]{bcg95} it holds
$$
\vol(M)=\deg(f) \cdot \vol(N) \ ,
$$
and hence $\vol(f^\ast\Psi)=\vol(M)$, as claimed. 

Assume now that $\vol(f^\ast \Psi)=\vol(M)$. By Inequality \eqref{eq:pullback:bound} we argue that
$$
\jac_a \widetilde F =1 \ ,
$$
for almost every $a \in \hypkn$. Thus $\widetilde F$ is an isometry and $F:M \rightarrow N$ is a local isometry. In particular it holds
$$
\vol(M)=\deg(F) \cdot \vol(N) = \deg(f) \cdot \vol(N) \ , 
$$
and the statement follows again by the strict version of mapping degree theorem. 
\end{proof}


\bibliographystyle{amsalpha}
\bibliography{biblionote}

\end{document}